\documentclass[12pt]{amsart}
\usepackage{amssymb,amsmath,amsthm}
\usepackage{enumerate}
\usepackage{xcolor}
\usepackage{blkarray}
\usepackage{mathrsfs}
\newtheorem{theorem}{Theorem}[section]
\newtheorem{proposition}[theorem]{Proposition}
\newtheorem{lemma}[theorem]{Lemma}
\newtheorem{corollary}[theorem]{Corollary}
\theoremstyle{definition}

\newtheorem{example}[theorem]{Example}
\newtheorem{remark}[theorem]{Remark}

\newcommand{\clc}{\mathcal{C}}

\newcommand{\cle}{\mathcal{E}}

\newcommand{\clq}{\mathcal{Q}}

\newcommand{\raro}{\rightarrow}
\newcommand{\NA}{\mathcal{N} \mathcal{A}}

\newcommand{\vp}{\varphi}

\newcommand{\bd}{\mathbb{D}}
\newcommand{\bt}{\mathbb{T}}

\usepackage[utf8]{inputenc}
\newcommand\hlight[1]{\tikz[overlay, remember picture,baseline=-\the\dimexpr\fontdimen22\textfont2\relax]\node[rectangle,fill=white!50,rounded corners,fill opacity = 0.2,draw,thick,text opacity =1] {$#1$};}

\title[Idempotent, model, and Toeplitz operators]{Idempotent, model, and Toeplitz operators attaining their norms}

\author[Bala]{Neeru Bala}
\address{Neeru Bala, Department of Mathematics, Indian Institute of Technology Bombay, Powai, Mumbai, 400076, India}
\email{ma16resch11001@iith.ac.in }

\author[Dhara]{Kousik Dhara}
\address{Kousik Dhara, Indian Statistical Institute,
Statistics and Mathematics Unit, 8th Mile, Mysore Road,
Bangalore, 560 059, India.}
\email{kousik.dhara1@gmail.com}

\author[Sarkar]{Jaydeb Sarkar}
\address{Jaydeb Sarkar, Indian Statistical Institute,
Statistics and Mathematics Unit, 8th Mile, Mysore Road,
Bangalore, 560 059, India.}
\email{jaydeb@gmail.com, jay@isibang.ac.in}

\author[Sensarma]{Aryaman Sensarma}
\address{Aryaman Sensarma, Indian Statistical Institute,
Statistics and Mathematics Unit, 8th Mile, Mysore Road,
Bangalore, 560 059, India.}
\email{aryamansensarma@gmail.com}


\numberwithin{equation}{section}

\begin{document}
\begin{abstract}
We study idempotent, model, and Toeplitz operators that attain the norm. Notably, we prove that if $\mathcal{Q}$ is a backward shift invariant subspace of the Hardy space $H^2(\mathbb{D})$, then the model operator $S_{\mathcal{Q}}$ attains its norm. Here $S_{\mathcal{Q}} = P_{\mathcal{Q}}M_z|_{\mathcal{Q}}$, the compression of the shift $M_z$ on the Hardy space $H^2(\mathbb{D})$ to $\mathcal{Q}$.
\end{abstract}
\maketitle
\section{Introduction}

Let $\cle$ be a Hilbert space (here all Hilbert spaces are separable and over $\mathbb{C}$) and $T$ be a bounded linear operator on $\cle$ ($T \in \mathscr{B}(\cle)$ in short). Then $T$ is said to be \textit{norm-attaining} (in short $T \in \NA$) if there exists a non-zero vector $f \in \cle$ such that
\[
\|Tf\|_{\cle} = \|T\|_{\mathscr{B}(\cle)} \|f\|_{\cle}.
\]
As far as the theory of bounded linear operators is concerned, it is perhaps very natural to study operators that attain the norm. It is also worth to point out that compact operators are always norm attaining. While the norm attaining property at the Banach space level has been studied extensively (for instance, see \cite{Acosta,L,M}), the same for operators on Hilbert spaces has so far received far less attention (however, see \cite{Neves, SKP, PAL, Ramesh}). On the other hand, in $1965$ Brown and Douglas \cite[Lemma 2]{BD}, in answering a question of H. Helson \cite[page 12]{HH}, established a close connection between arithmetic of inner functions, Toeplitz operators, and norm attaining operators on Hilbert spaces. Curiously, this is also intimately connected with the Sarason's commutant lifting theorem \cite{Sarason}.

In this paper, we study norm attainment of three classical (and fairly non-compact in nature) Hilbert space operators, namely, Toeplitz operators, model operators, and idempotent operators. Toeplitz operators are one of the most important concrete operators, where concept of a model operator is one of the most useful in operator theory and function theory having important applications in various fields. Model operators also play the role of building blocks in the basic theory of linear operators \cite{Hari, NF}. Idempotent operators (also known as oblique projections) are yet another concrete (but complex) class of operators that plays a significant role in many definite problems in operator theory and operator algebras.

In Section \ref{sec: ideompotents} we study idempotent operators. Let $T$ be an idempotent operator (that is, $T^2 = T$) on some Hilbert space. In Theorem \ref{prop: NA idempotent} we prove that $T \in \NA$ if and only if the \textit{Buckholtz operator} $T+T^*-I$ is in $\NA$. Observe that the Buckholtz operator \cite{Buc1, Buc} is a self-adjoint operator.

Section \ref{sec:MO} deals with model operators. Let $M_z$ denote the forward shift (or the multiplication operator by the coordinate function $z$) on the Hardy space $H^2(\bd)$. Let $\clq$ be a closed $M_z^*$-invariant subspace of $H^2(\bd)$ (see Section \ref{sec:MO} for more details). The \textit{model operator} $S_{\clq}$ is the compression of $M_z$ to $\clq$, that is, $S_{\clq} = P_{\clq} M_z|_{\clq}$. Theorem \ref{thm: norm 1} says that
\[
S_{\clq} \in \NA.
\]
Of course, the above $\clq$ is associated with an (essentially unique) inner function $\theta \in H^\infty(\bd)$ (because of Beurling), that is, $\clq = \clq_\theta$, where $\clq_\theta = H^2(\bd)/\theta H^2(\bd)$. The representing  Beurling inner function $\theta$ plays an essential role in the proof of the above result. We also obtain norm attaintment result for the general Sz.-Nagy and Foia\c{s} model operators (vector-valued counterpart of $S_\clq$).

The final section, Section \ref{sec: TO}, of this paper deals  with Toeplitz and analytic Toeplitz operators. In Theorem \ref{Toeplitz Operator Valued} we prove that a Toeplitz operator $T_\Phi$ with operator-valued symbol $\Phi$ is in $\NA$ if and only if that $\Phi$ satisfies certain inner function divisibility criterion. This result recovers Brown and Douglas classification of norm attaining Toeplitz operators with scalar-valued symbols. We also discuss the case of analytic Toeplitz operators (see Theorem \ref{mult op Vector valued}) and Laurent operators (see Proposition \ref{NA scalar Laurent }). Examples \ref{not NA eg} and \ref{eg: back shift} bring out more insight between scalar and vector-valued Toeplitz operators.

Before we proceed to the main content of this paper, in the following we collect some useful results (see \cite[Corollary 2.4, Proposition 2.5]{Neves} and \cite[Theorem 2.4]{PAL}):

\begin{theorem}\label{thm: all results}
Let $T \in \mathscr{B}(\cle)$. The following are equivalent:

(i) $T \in \NA$.

(ii) $T^* \in \NA$.

(iii) $TT^* \in \NA$.

(iv) $\|T\|^2$ is in the point spectrum of $T T^*$.
\end{theorem}

This will be used frequently in what follows.

\section{Idempotent Operators}\label{sec: ideompotents}

It is evident that any orthogonal projection on a Hilbert space is norm attaining. Here we deal with the issue of norm attainment of idempotent operators. We begin with an example of an idempotent which is not a norm attaining operator.

\begin{example}\label{example idempotent}
Define a linear operator $T: \ell^2(\mathbb{N}) \longrightarrow \ell^2(\mathbb{N})$ by
\begin{align*}
T(\{\alpha_n\}_{n=1}^\infty) = \Big\{\alpha_1, 0, \alpha_3, (1-\dfrac{1}{3}) \alpha_3, \alpha_5, (1-\dfrac{1}{5}) \alpha_5, \ldots, \alpha_{2n+1}, (1-\dfrac{1}{2n+1})\alpha_{2n+1},\ldots\Big\}.
\end{align*}
Clearly, $T$ is an idempotent operator. Note that for any $\alpha = \{\alpha_n\}_{n=1}^\infty \in \ell^2(\mathbb{N})$, we have
\begin{align}\label{Eqn 1}
\|T\alpha\|^2 = \sum_{n=0}^{\infty}\left \{1+\left(1-\dfrac{1}{2n+1}\right)^2\right\}|\alpha_{2n+1}|^2 < 2 \|\alpha\|^2,
\end{align}
and hence, $\|T\| \leq \sqrt{2}$.
Furthermore, for all $n \geq 1$, we also have
\[
\|T(e_{2n+1})\|^2 =1+\left(1-\dfrac{1}{2n+1}\right)^2,
\]
which implies that
\[
\|T\| \geq \sqrt{1+\left(1-\dfrac{1}{2n+1}\right)^2},
\]
and hence $\|T\| \geq \sqrt{2}$. We conclude that $\|T\|=\sqrt{2}$. Finally, by $\eqref{Eqn 1}$, it follows that $T \notin \NA$.
\end{example}

We turn now to classifying idempotent operators that admit the norm. We begin by recalling a geometric construction of idempotent operators. Let $T \in \mathscr{B}(\cle)$ be an idempotent.
Let $P$ denote the orthogonal projection onto $\mbox{ran~} T$. Then, by Feldman, Krupnik and Markus \cite[Equation (1.8)]{FELDMAN} (also see \cite{Bottcher,Szyld}), there exists $X\in\mathscr{B}(\mbox{ker~} P,\mbox{ran~} P)$ such that
\[T=
\begin{bmatrix}
I& X\\
0&0
\end{bmatrix},
\]
on $\mbox{ran~} P \oplus \ker P$. Set $A=I+XX^*$. Then
\[
TT^*=\begin{bmatrix}
A&0\\
0&0
\end{bmatrix}.
\]
Note that $\|A\|=1+\|X\|^2$ and $\|T\|^2=\|TT^*\|=\|A\|$.

We are now ready to prove our first classification result. Observe that the \textit{Buckholtz operator} $T+T^*-I$ is a self adjoint operator.

\begin{theorem}\label{prop: NA idempotent}
Let $T\in\mathscr{B}(\cle)$ be an idempotent. Then $T \in \NA$ if and only if
\[
T+T^*-I \in \NA.
\]
\end{theorem}
\begin{proof}
We continue to use the above notation. Set $B=I+X^*X$. It is easy to see that $T+T^*-I=\begin{bmatrix}I&X\\X^*&-I\end{bmatrix}$, and hence 
\[
(T+T^*-I)(T+T^*-I)^*=(T+T^*-I)^2=\begin{bmatrix}
A&0\\
0&B
\end{bmatrix}.
\]
Since $\|B\|=\|A\|=1+\|X\|^2$, it readily follows that 
\[
\|T+T^*-I\|^2=1+\|X\|^2.
\]	
Now let $T \in \NA$. By Theorem \ref{thm: all results}, we infer that $\|T\|^2$ is an eigenvalue of $TT^*$. Then there exists a non-zero vector $f \in \mbox{ran~} P$ such that $Af=\|A\|f$. It follows that
\[\begin{bmatrix}
	A&0\\
	0&B
\end{bmatrix}\begin{bmatrix}
f\\
0
\end{bmatrix}=(1+\|X\|^2)\begin{bmatrix}
f\\
0
\end{bmatrix},\]
 which, along with Theorem \ref{thm: all results} implies that $T+T^*-I \in \NA$.

\noindent Conversely, assume that $T+T^*-I\in\NA$. By Theorem \ref{thm: all results}, there exists a non-zero vector 
\[
\begin{bmatrix}
f\\
g
\end{bmatrix} \in \mbox{ran~} P \oplus \ker P.
\]
such that
\begin{equation}\label{eqn1 eigenvalue}
Af = (I+XX^*)f=\left(1+\|X\|^2\right)f,
\end{equation}
and 
\begin{equation}\label{eqn2 eigenvalue}
Bg = (I+X^*X)g=\left(1+\|X\|^2\right)g.
\end{equation}
If $f\ne 0$, then the matrix representation of $TT^*$ and \eqref{eqn1 eigenvalue} imply that $T\in\NA$. Now assume that $f=0$ and $g\ne 0$. Multiplying \eqref{eqn2 eigenvalue} from left by $X$ we get
\begin{equation}\label{eigenvalue of T}
	(I+XX^*)Xg=(1+\|X\|^2)Xg.
\end{equation}
If $Xg=0$, then \eqref{eqn2 eigenvalue} implies that $g=g+\|X\|^2g$. Since $g\ne 0$, we get $X=0$. Hence $T$ is an orthogonal projection and, consequently, $T\in\NA$. On the other hand, if $Xg\ne 0$, then by \eqref{eigenvalue of T}, we obtain 
\[
AXg=\|A\|Xg.
\]
Then, the matrix representation of $TT^*$ together with Theorem \ref{thm: all results} asserts that $T\in\NA$. This completes the proof of the theorem.	
\end{proof}

\begin{remark}
The present proof of Theorem \ref{prop: NA idempotent} is due to the referee which is more elegant than our original proof. Our original proof here had been based on the argument of Ando \cite[Theorem 2.6, Theorem 3.9]{Ando}.
\end{remark}

The proof of the corollary below now follows easily from Theorem \ref{prop: NA idempotent} and Theorem \ref{thm: all results}:

\begin{corollary}\label{coro idempotent}
Let $T\in\mathscr{B}(H)$ be an idempotent operator. Then $T \in \NA$ if and only if either $\|T\|$ or $-\|T\|$ is an eigenvalue of $T+T^*-I$.
\end{corollary}

In particular, an idempotent operator is norm attaining if and only if the corresponding Buckholtz operator is norm attaining. Now, we return to the idempotent $T$ in Example \ref{example idempotent} and validate the above corollary. First note that
\[
T^*(\{\alpha_n\}_{n=1}^\infty)= \Big\{\alpha_1,0,\alpha_3+\left(1-\frac{1}{3}\right)\alpha_4,0,\alpha_5+\left(1-\frac{1}{5}\right)\alpha_6,0,\ldots\Big\}.
\]
Then
\[
(T+T^*-I)(\{\alpha_n\}_{n=1}^\infty) = \Big\{\alpha_1,-\alpha_2,\alpha_3+(1-\frac{1}{3})\alpha_4,(1-\frac{1}{3})\alpha_3-\alpha_4, \ldots\Big\}.
\]
If $(T+T^*-I)(\{\alpha_n\}_{n=1}^{\infty})=\pm\sqrt{2}(\{\alpha_n\}_{n=1}^{\infty})$ for some $\{\alpha_n\}_{n=1}^{\infty}\in\ell^2(\mathbb{N})$, then the above implies that $\alpha_n=0$ for all $n$. By Corollary \ref{coro idempotent}, we conclude that $T \notin \NA$.

We also have the following general result:
\begin{theorem}
Let $T\in\mathscr{B}(\cle)$ be an idempotent operator. Then $T \in \NA$ if and only if there exists $f \in \cle$ such that $Tf \neq 0$, $P_{\mbox{ran~}T^*} T f = f$, and
\[
T^*P_{(\mbox{ran~}T^*)^{\perp}}T f = (\|T\|^2-1)f.
\]
\end{theorem}

\begin{proof}
Suppose $T \in \NA$. By Theorem \ref{thm: all results}, there exists a non-zero vector $h \in \cle$ such that $TT^* h = \|T\|^2h$. Hence
\[
\|T\|^2h = TT^*h = T^2T^*h = T(\|T\|^2h) = \|T\|^2 Th,
\]
which implies that $Th = h$. Set $h = f \oplus g \in \mbox{ran~} T^* \oplus (\mbox{ran~} T^*)^{\perp}$. Since $T(f + g) = f  +g$, $g \in (\mbox{ran~}T^*)^{\perp} = \ker~ T$, we get $Tf = f + g$, and hence we obtain
\[
P_{\mbox{ran~}T^*} Tf = f \; \mbox{ and } \; P_{(\mbox{ran~}T^*)^{\perp}} T f = g.
\]
Now $TT^*(f+g)=\|T\|^{2}(f+g)$ and $f \in \mbox{ran~}T^*$ implies that
\[
\|T\|^{2}(f+g) = Tf + T T^* g = f + g + T T^* g,
\]
and hence $T T^* g = (\|T\|^{2}-1) (f+g)$. Then $P_{\mbox{ran~}T^*} (T T^* g) = (\|T\|^{2}-1) f$, which, along with $P_{\mbox{ran~}T^*} Tf = f $ implies that $P_{\mbox{ran~}T^*} T (T^* g - (\|T\|^{2}-1) f) = 0$. Thus
\[
T (T^* g - (\|T\|^{2}-1) f) \in (\mbox{ran~}T^*)^\perp = \ker~ T,
\]
and hence $T (T^* g - (\|T\|^{2}-1) f)= 0$ as $T^2 = T$. Then $T^* g - (\|T\|^{2}-1) f \in (\mbox{ran~}T^*)^\perp$, where, on the other hand, $f \in \mbox{ran~}T^*$ implies that $T^* g - (\|T\|^{2}-1) f \in \mbox{ran~}T^*$. This is possible only when
\[
T^* g - (\|T\|^{2}-1) f = 0,
\]
which, along with $P_{(\mbox{ran~}T^*)^{\perp}} T f = g$, implies $T^* P_{(\mbox{ran~}T^*)^{\perp}} T f = (\|T\|^{2}-1) f$.

\noindent Conversely, assume $Tf \neq 0$ for some $f \in \cle$, and assume that $P_{\mbox{ran~}T^*} T f = f$ and $T^* P_{(\mbox{ran~}T^*)^{\perp}} T f = (\|T\|^{2}-1) f$. Set $g = P_{(\mbox{ran~}T^*)^{\perp}} T f$. Then $Tf = f + g$ and $T^* g = (\|T\|^{2}-1) f$. Since
\[
TT^*T f = TT^*f + TT^*g = Tf + TT^*g,
\]
we have $TT^*Tf = \|T\|^2 T f$. Then Theorem \ref{thm: all results} implies that $T \in \NA$, and completes the proof of the theorem.
\end{proof}

In particular, an idempotent operator $T$ is in $\NA$ if and only if $1$ and $\|T\|^2-1$ are eigenvalues of $P_{\mbox{ran~}T^*} T$ and $T^*P_{(\mbox{ran~}T^*)^{\perp}}T$, respectively, corresponding to a common eigenvector.

\section{Model Operators}\label{sec:MO}

In this section we will treat model operators that attain the norm. We begin with some standard terminology. Let $\bd$ be the open unit disc in $\mathbb{C}$, and let $\cle$ be a Hilbert space. The $\cle$-valued \textit{Hardy space} $H^2_\cle(\bd)$ (or $H^2(\bd)$ if $\cle = \mathbb{C}$ ) over $\bd$ is the Hilbert space of all $\cle$-valued analytic functions $f = \sum\limits_{n=0}^{\infty} a_n z^n$, $a_n \in \cle$, on $\bd$ such that
\[
\|f\| = (\sum_{n=0}^{\infty} \|a_n\|^2)^{\frac{1}{2}} < \infty.
\]
Given another Hilbert space $\cle_*$, we denote by $H^\infty_{\mathscr{B}(\cle_*, \cle)}(\bd)$ (or $H^\infty_{\mathscr{B}(\cle)}(\bd)$ if $\cle_* = \cle$) the set of $\mathscr{B}(\cle_*, \cle)$-valued bounded analytic functions on $\bd$. Also recall that a function $\Theta\in H^\infty_{\mathscr{B}(\cle_*, \cle)}(\bd)$ is called \textit{inner} if $\Theta(z)$ (via radial limits) is an isometry for all $z$ a.e. in $\bt$. If $\cle = \cle_* = \mathbb{C}$, then $H^\infty_{\mathscr{B}(\cle_*, \cle)}(\bd)$ will be denoted by $H^\infty(\bd)$. In particular, a function $\vp \in H^\infty(\bd)$ is inner if and only if $|\vp(z)| = 1$ for all $z$ a.e. in $\bt$.

Let $\cle$ be a Hilbert space, and let $\clq \subseteq H^2_{\cle}(\bd)$ be a closed subspace. We say that $\clq$ is a \textit{Sz.-Nagy and Foia\c{s} model space} if $\clq$ is $M_z^*$-invariant. In this case, the \textit{Sz.-Nagy and Foia\c{s} model operator} $S_\clq$ is the compression of $M_z$ on $\clq$, that is
\[
S_\clq = P_{\clq} M_z|_{\clq}.
\]
If $\cle = \mathbb{C}$, we simply say that $\clq$ is a \textit{model space} and $S_\clq$ is a \textit{model operator}. In this paper we always assume that $\{0\} \subsetneq \clq \subsetneq H^2_{\cle}(\bd)$. Observe that Sz.-Nagy and Foia\c{s} model operators essentially represents the set of all contractions $T$ such that $T^{*n} \raro 0$ in the strong operator topology \cite{NF}.

\begin{proposition}\label{prop: NA Q}
Suppose $\clq \subseteq H^2_{\cle}(\bd)$ be a Sz.-Nagy and Foia\c{s} model space. Then there exists a non-zero vector $f\in \clq$ such that $f(0)=0$ if and only if $S_\clq \in \NA$ and $\|S_\clq\|=1$.
\end{proposition}

\begin{proof}
If we set $\clc := \{h \in \clq: \|S_\clq h\|=\|h\| \}$, then
\[
\clc =\{ h\in \clq: zh \in \clq\}.
\]
Indeed, if $h \in \clc$, then
\[
\|h\| = \|S_\clq h\| = \|P_{\clq} M_z h\| \leq \|M_z h\| = \|h\|,
\]
and hence $\|P_{\clq} (zh)\| = \|zh\|$. This implies that $zh \in \clq$, as $P_{\clq}$ is an orthogonal projection of $H^2_{\cle}(\bd)$ onto $\clq$. On the other hand, if $h \in \clq$ and $zh \in \clq$, then clearly
\[
\|S_\clq h\| = \|P_{\clq} zh\| = \|M_zh\| = \|h\|.
\]
Now, let $\|S_\clq\|=1$ and $S_\clq \in \NA$. This implies that $\clc \neq \{0\}$. Then there exists a non-zero vector $h \in \clq$ such that $f:=zh \in \clq_\theta$. Clearly, $f(0) = 0$. Conversely, let $f \in \clq$ be a non-zero vector such that $f(0) = 0$. Then $f= zg \in \clq$ for some $g \in H^2_{\cle}(\bd)$. Since $\clq$ is $M_z^*$-invariant, it follows that $g = M_z^* f \in \clq$. Then
\[
\|S_\clq g\|=\|P_{\clq}(zg)\|=\|zg\|=\|g\|,
\]
which implies that $\|S_\clq\| = 1$ and $S_\clq \in \NA$.
\end{proof}

We now concentrate on the special case when $\clq$ is a model space, that is, $\clq \subseteq H^2(\bd)$. First, we recall that the Hardy space $H^2(\bd)$ is also a reproducing kernel Hilbert space corresponding to the Szeg\H{o} kernel $c$ on $\bd$, where
\[
c(z, w) = (1 - z \bar{w})^{-1} \qquad (z, w \in \bd).
\]
Then the set of \textit{kernel functions} $\{c(\cdot, w): w \in \bd\}$ forms a total set in $H^2(\bd)$, and satisfies the \textit{reproducing property} $f(w) = \langle f, c(\cdot, w)\rangle_{H^2(\bd)}$, for all $w \in \bd$ and $f \in H^2(\bd)$. Now suppose that $\clq$ is a model space. Then the classical Beurling theorem yields
\[
\clq = \clq_\theta: = H^2(\bd)/ \theta H^2(\bd),
\]
where $\theta \in H^\infty(\bd)$ is an inner function (which is unique up to multiplication by a scalar of modulus one). Then the corresponding model operator $S_\clq$, denoted by $S_\theta$, is given by $S_\theta = P_{\clq_\theta} M_z|_{\clq_{\theta}}$, where $P_{\clq_\theta}$ is the orthogonal projection of $H^2(\bd)$ onto $\clq_\theta$. One can easily prove that
\[
c_\theta(z,w) = \frac{1 - \theta(z) \overline{\theta(w)}}{1 - z \bar{w}} \qquad (z, w \in \bd),
\]
defines the reproducing kernel function of $\clq_\theta$. In particular
\[
c_\theta(z, 0) = 1 - \theta(z) \overline{\theta(0)},
\]
and hence
\[
\|c_\theta(\cdot, 0)\| = (1 - |\theta(0)|^2)^\frac{1}{2}.
\]
The following result complements Proposition \ref{prop: NA Q}:

\begin{proposition}\label{NA Model Op}
Let $\theta \in H^\infty(\bd)$ be inner. Then $\|S_\theta\|=|\theta(0)|$ if and only if $\|S_\theta\|<1$ and $S_\theta \in \NA.$
\end{proposition}

\begin{proof}
Suppose $\|S_\theta\|<1$ and $S_\theta \in \NA$. Then, by Theorem \ref{thm: all results}, there exists a non-zero vector $f \in \clq_\theta$ such that
\[
(\|S_\theta\|^2 I_{\clq_\theta} - S_\theta S_\theta ^*)f = 0.
\]
It is easy to see that $S_\theta S_\theta^* = I_{\clq_\theta} - c_\theta(\cdot, 0) \otimes c_\theta(\cdot, 0)$. Then
\begin{equation}\label{eqn: proof 1}
(\|S_\theta\|^2-1) f + \langle f, c_\theta(\cdot, 0) \rangle c_\theta(\cdot, 0) = 0.
\end{equation}
Taking inner product with $c_\theta(\cdot, 0)$, we have
\[
(\|S_\theta\|^2 - 1) \langle f, c_\theta(\cdot, 0) \rangle + \left(1- |\theta (0)|^2 \right) \langle f, c_\theta(\cdot, 0) \rangle =0,
\]
as $\|c_\theta(\cdot, 0)\|^2 = 1 - |\theta(0)|^2$. Note that since $\|S_\theta\|<1$ and $f\neq 0$, \eqref{eqn: proof 1} implies that $\langle f, c_\theta(\cdot, 0) \rangle \neq 0$. Then we have $\|S_\theta\|^2-1+ \left(1- |\theta (0)|^2 \right)=0$, and hence $\|S_\theta\| = |\theta(0)|$. Conversely, if $\|S_\theta\|=|\theta(0)|$, then
\begin{align*}
(\|S_\theta\|^2 I_{\clq_\theta} - S_\theta S_\theta ^*) c_\theta(\cdot, 0) & = \Big(|\theta(0)|^2 I_{\clq_\theta} - \left(I_{\clq_\theta}- c_\theta(\cdot, 0) \otimes c_\theta(\cdot, 0) \right) \Big)c_\theta(\cdot, 0)
\\
&= |\theta(0)|^2 c_\theta(\cdot, 0) - c_\theta(\cdot, 0) + \|c_\theta(\cdot, 0)\|^2 c_\theta(\cdot, 0)
\\
&= (|\theta(0)|^2 - 1 + (1-|\theta(0)|^2)) c_\theta(\cdot, 0)
\\
&=0,
\end{align*}
that is, $\|S_\theta\|^2$ is in the point spectrum of $S_\theta S_\theta ^*$. This along with Theorem \ref{thm: all results} shows that $S_\theta \in \NA$. Finally, since $\theta$ is inner, by the maximum modulus principle we conclude that $1> |\theta(0)| = \|S_\theta\|$. This completes the proof.
\end{proof}

We now proceed to the most definite result of this section. For any $\lambda \in \bd$, we denote by $b_{\lambda}$ the \textit{Blaschke factor} corresponding to $\lambda$, that is
\[
b_{\lambda}(z) = \frac{z-\lambda}{1 - \bar{\lambda} z} \qquad (z \in \bd).
\]

\begin{theorem}\label{thm: norm 1}
Let $\theta \in H^\infty(\bd)$ be an inner function. Then $S_\theta \in \NA$. Moreover, $\|S_\theta\| = 1$ if and only if $\mbox{dim} \clq_\theta > 1$.
\end{theorem}
\begin{proof}
First we consider $\mbox{dim} \clq_\theta < \infty$. Suppose $\mbox{dim} \clq_\theta = 1$. Then there exists $\lambda \in \bd$ such that $\theta = b_{\lambda}$ and $\clq_\theta = \mathbb{C} c(\cdot, \lambda)$. Then $S_\theta^* c(\cdot, {\lambda}) = \bar{\lambda} c(\cdot, \lambda)$ implies that
\[
\|S_\theta\|= |\lambda|=|\theta(0)| < 1.
\]
Now let $\mbox{dim} \clq_\theta = n (>1)$. Then $\theta = \prod\limits_{i=1}^{n} b_{\lambda_i}$ for some $\{\lambda_i\}_{i=1}^n \subseteq \bd$. Suppose $\lambda_p \neq \lambda_q$ for some $p \neq q$ and $1 \leq p,q \leq n$, and suppose
\[
f = c(\cdot, {\lambda_p}) - c(\cdot, {\lambda_q}).
\]
If $\lambda = \lambda_i$ for all $i =1, \ldots, n$, then we consider $f$ as
\[
f = \overline{\lambda}c(\cdot, \lambda) - \partial c(\cdot, {\lambda}),
\]
where the partial derivative is with repect to the first variable.
In either case, $f \in \clq_\theta$ and $f(0) = 0$, and thus, by Proposition \ref{prop: NA Q}, $\|S_\theta\| = 1$ and $S_\theta \in \NA$ (of course, the latter conclusion is trivial as $S_\theta$ is a finite rank operator).

\noindent We now consider the infinite dimensional case. Assume first that $\theta$ is an infinite Blaschke product, that is
\[
\theta(z)=\prod_{n=1}^{\infty} \alpha_n b_{\lambda_n} (z) \qquad (z \in \bd),
\]
where $\sum_{k=1}^{\infty} (1-|\lambda_k|)< \infty$ and $|\alpha_n|=1$, for all $n \geq 1$. Then $\clq_\theta= \overline{\text{span}} \{ c(\cdot, {\lambda_n}): n \geq 1\}$. Since $S_\theta^* c(\cdot, {\lambda_m}) = \bar{\lambda}_m c(\cdot, {\lambda_m})$ for all $m \geq 1$, and $|\lambda_n| \rightarrow 1$ as $n \rightarrow \infty$, it follows that
\[
1 = \underset{n\in \mathbb{N}}{\sup} |\lambda_n | \leq \|S_\theta\| \leq 1,
\]
which implies that $\|S_\theta\|=1$. Again we set $f = c(\cdot, {\lambda_p}) - c(\cdot, {\lambda_q})$ or $f = \overline{\lambda}c(\cdot, \lambda) - \partial c(\cdot, {\lambda})$, as above, and conclude that $f(0) = 0$ for some $f \in \clq_\theta$. Then Proposition \ref{prop: NA Q} again implies that $S_\theta \in \NA$.

\noindent Now assume that $\theta$ is a singular inner function. By Frostman's theorem, there exist distinct $\lambda_1$ and $\lambda_2$ in $\bd$ such that $c_\theta(\cdot, \lambda_1), c_\theta(\cdot, \lambda_2) \in \clq_\theta$. Set
\[
f= \Big(1 - \overline{\theta(\lambda_2)} \theta(0)\Big) c_\theta(\cdot, \lambda_1) - \Big(1 - \overline{\theta(\lambda_1)} \theta(0)\Big) c_\theta(\cdot, \lambda_2).
\]
Then $f  \in \clq_\theta$ and $f(0)=0$, as $c_\theta(0, w) = 1 - \theta(0) \overline{\theta(w)}$, $w \in \bd$. By Theorem \ref{NA Model Op} it follows that $S_\theta \in \NA$ and $||S_\theta\|=1$.

\noindent Finally, let $\theta= \theta_b \theta_s$, where $\theta_b$ is the Blaschke product formed by the zeros of $\theta$ and $\theta_s$ is the corresponding singular factor of $\theta$. Since $\theta_s | \theta$, it easily follows that (by the Douglas range inclusion theorem)
\[
\clq_{\theta_s} \subseteq \clq_\theta.
\]
We can then use the singular inner function part above to find a $f \in \clq_{\theta_s} \subseteq \clq_\theta$ such that $f(0)=0$. Consequently, $S_\theta \in \NA$ and $\|S_\theta\|=1$. This completes the proof of the theorem.
\end{proof}

\begin{remark}\label{rem}
The latter conclusion of Theorem \ref{thm: norm 1} is not new, and it essentially follows from \cite[Section 7, Corollary 3]{GR}. In fact, \cite[Section 7]{GR} deals with the problem of norm attaining symbols of truncated Toeplitz operators: given $\vp \in L^\infty(\bt)$, when does $\|A_\vp^\theta\|_{\mathscr{B}(\clq_\theta)} = \|\vp\|_{\infty}$, where $A_\vp^\theta = P_{\clq_\theta} L_\vp|_{\clq_\theta}$ (see Section \ref{sec: TO}) is the truncated Toeplitz operator. This problem, a priori, is different from our norm attaining operators. On the other hand, our approach, like divisibility of functions in model spaces as in the proof of Proposition \ref{prop: NA Q}, certainly relies on classical technique as in \cite{BD} and \cite{GR}.
\end{remark}

\section{Toeplitz Operators}\label{sec: TO}

In this section we consider Toeplitz and Laurent operators that admit the norm. We first introduce the classical vector-valued Hilbert  measure spaces. Suppose $\cle$ is a Hilbert space. Let $L^2_\cle(\bt)$ (here $\bt = \partial \bd$) denote the Hilbert space of all square $\cle$-valued (Lebesgue) integrable functions on $\bt$, that is
\[
L^2_\cle(\bt) = \left \{ f:\bt \rightarrow \cle \text{ measurable}: \|f\|= \Big[\int_{\bt} \|f(z)\|^2_\cle\, dm(z)\Big]^{\frac{1}{2}} <\infty \right \},
\]
where $m$ is the normalized Lebesgue measure on $\bt$. The Hardy space $H^2_\cle(\bd)$ also can be identified (via radial limits) to the subspace (which we will denote again by $H^2_\cle(\bd)$) of $\cle$-valued functions $f$ in $L^2_\cle(\bt)$ such that $\hat{f}(n) = 0$, $n<0$, where $\hat{f}(n)$ is the $n$-th Fourier coefficient of $f$. Given another Hilbert space $\cle_*$, we denote by $L^\infty_{\mathscr{B}(\cle_*, \cle)}(\bt)$ (or $L^\infty_{\mathscr{B}(\cle)}(\bt)$ if $\cle_* = \cle$) the set of $\mathscr{B}(\cle_*, \cle)$-valued bounded functions on $\bt$.

Now we turn to the main content of this section. We begin with Laurent operators: Let $ \cle_*$ and $\cle$ be Hilbert spaces. For $\Phi \in L^\infty_{\mathscr{B}(\cle_*, \cle)}(\bt)$, the \textit{Laurent operator} $L_\Phi : L^2_{\cle_*}(\bt) \rightarrow L^2_\cle(\bt)$ is defined by $(L_\Phi f)(z) = \Phi(z) f(z)$, $z \in \bt$. In this case, $L_\Phi$ is bounded and $\|L_\Phi\| = \|\Phi\|_{\infty}$. The \textit{Toeplitz operator} $T_\Phi : H^2_{\cle_*} (\bd) \rightarrow H^2_\cle (\bd)$ with (operator-valued) symbol $\Phi$ is defined by
\[
T_\Phi = P_{H^2_\cle (\bd)}L_\Phi|_{H^2_{\cle_*} (\bd)},
\]
where $P_{H^2_\cle (\bd)}$ is the orthogonal projection of $L^2_\cle(\bt)$ onto $H^2_\cle (\bd)$. In particular, when $\cle_*=\cle$, $L_z$ is the bilateral shift, and $T_z = P_{H^2_\cle (\bd)}L_z|_{H^2_\cle (\bd)} = M_z$, as the symbol $z$ is analytic. It is well known that $\|T_\Phi\|=\|\Phi\|_\infty$ (cf. \cite[Theorem 1.7, page 112]{Hari}).

The following theorem provides a complete characterization of norm attaining operator valued Toeplitz operators.

\begin{theorem}\label{Toeplitz Operator Valued}
Let $\Phi \in L^\infty_{\mathscr{B}(\cle)}(\bt)$, and suppose $\|T_\Phi\|=1$. Then $T_\Phi \in \NA$ if and only if there exist a Hilbert space $\cle_*$ and inner functions $\Theta , \Psi \in H^\infty_{\mathscr{B}(\cle_*, \cle)}(\bd)$ such that $\Theta = \Phi \Psi$. Moreover, in this case $T_\Theta = T_\Phi T_\Psi$.	
\end{theorem}
\begin{proof}
As in the proof of Proposition \ref{prop: NA Q}, we set $\clc = \{ h\in H^2_\cle (\bd): \|T_\Phi h\|=\|h\|\}$. If $h \in \clc$, then
\[
\|h\| = \|P_{H^2_\cle (\bd)} (\Phi h)\| \leq \|\Phi h\| \leq \|h\|,
\]
and hence $\|P_{H^2_\cle (\bd)}(\Phi h)\|=\|\Phi h\|$, or, equivalently, $\Phi h \in H^2_\cle (\bd)$. In particular, $\|L_\Phi h\| = \|\Phi h\| = \|h\|$ implies $\langle (I - L_\Phi^* L_\Phi)h, h \rangle  =0$, and hence $L_\Phi ^* L_\Phi h= h$. As the reverse direction is obvious, we have
\[
\clc = \{h \in H^2_\cle(\bd): \Phi h \in H^2_\cle(\bd) \text{ and } L_\Phi^* L_\Phi h=h\}.
\]
In particular, $\clc$ is a closed subspace of $H^2_\cle(\bd)$. Moreover, if $h \in \clc$ and $f \in L^2_\cle(\bt) \ominus H^2_\cle(\bd)$, then
\[
\|L_\Phi (zh)\| = \|L_\Phi L_z h\| = \|L_\Phi h \| = \|h\| = \|zh\|,
\]
and
\[
\langle \Phi z h, f \rangle_{L^2_\cle(\bt)} = \langle \Phi h, L_z^* f \rangle_{L^2_\cle(\bt)} = 0,
\]
implies that $\clc$ is an $M_z$-invariant subspace of $H^2_\cle(\bd)$.

\noindent Now suppose $T_\phi \in \NA$. Then $\clc \neq \{0\}$, and hence by Beurling, Lax and Halmos theorem \cite[Chapter V, Theorem 3.3]{NF}, there exist a Hilbert space $\cle_*$ and an inner function $\Psi \in  H^\infty_{\mathscr{B}(\cle_*, \cle)}(\bd)$ such that $\clc = \Psi H^2_{\cle_*}(\bd)$. Moreover, $\Phi \clc \subseteq H^2_{\cle}(\bd)$ implies that
\[
\Phi \Psi H^2_{\cle_*}(\bd) \subseteq H^2_{\cle}(\bd).
\]
Evidently, there exists $\Theta \in H^\infty_{\mathscr{B}(\cle_*, \cle)}(\bd)$ such that $\Theta = \Phi \Psi$. Moreover, if $f \in H^2_{\cle_*}(\bd)$, then
\[
\|\Theta f\| = \|\Phi \Psi f\| = \|\Psi f \| = \|f\|,
\]
which implies that $\Theta \in H^\infty_{\mathscr{B}(\cle_*, \cle)}(\bd)$ is an inner function.

\noindent For the converse, observe that since $\Theta$ and $\Psi$ are inner, for each $f \in H^2_{\cle_*}(\bd)$ we have
\[
\|T_\Phi(\Psi f)\| = \| P_{H^2_\cle (\bd)} \Phi \Psi f \| = \| P_{H^2_\cle (\bd)} \Theta f\| = \|\Theta f\| = \|f\| = \|\Psi f\|,
\]
which implies that $T_\Phi \in \NA$.

\noindent The final part is standard: $T_\Theta = T_\Phi T_\Psi$ follows from the fact that  $\Theta$ and $\Psi$ are inner and $\Theta = \Phi \Psi$.
\end{proof}

If $\cle = \mathbb{C}$, then the above theorem reduces to the Brown and Douglas \cite[Lemma 2]{BD}\label{Brown-Douglas lemma} classification of norm attaining Toeplitz operators with scalar-valued symbols:

\begin{corollary}
Let $\vp \in L^\infty(\bt)$ and suppose $\|\vp\|_\infty = 1$. Then $T_\vp \in \NA$ if and only if there exist inner functions $\psi, \theta \in H^\infty(\bd)$ such that $T_\vp = T_\psi^* T_\theta$.
\end{corollary}
\begin{proof}
In this case, $\cle_* = \mathbb{C}$. The result now follows from the observation that $T_\vp T_\psi = T_\psi T_\vp$.
\end{proof}

Now we turn to norm attaining analytic Toeplitz operators. Let $\cle$ be a Hilbert space, and let $\Phi \in H^\infty_{\mathscr{B}(\cle)}(\bd)$. The \textit{analytic Toeplitz operator} (or \textit{multiplication operator}) $M_\Phi:H^2_\cle(\bd) \rightarrow H^2_\cle(\bd)$ with symbol $\Phi$ is defined by
\[
(M_\Phi f)(z) = \Phi(z) f(z) \qquad (f\in H^2_\cle(\bd), z\in \bd).
\]
It is known that $\|M_\Phi\| = \|\Phi\|_\infty$, and $M_\Phi$ is an isometry if and only if $\Phi$ is inner \cite[Proposition 2.2]{NF}.

\noindent The following vector-valued analogue of F. and M. Riesz theorem is certainly well known, but we have not been able to trace an explicit reference in the literature.

\begin{lemma}\label{Risez thm vector valued}
If $f$ is a non-zero function in $H^2_\cle (\bd)$, then the measure of the set $\left \{z \in \bt: f(z)=0  \right \}$ is zero.
\end{lemma}
\begin{proof}
Let $f \in H^2_\cle (\bd)$ and suppose $E = \{z: f(z) = 0\}$, and $f(z)=\sum_{k=0}^{\infty} a_kz^k$, $z \in \bd$. For each $\eta\in \cle$ we define $f_\eta: \bd \rightarrow \mathbb{C}$ by $f_\eta(z)= \langle f(z), \eta \rangle_{\cle}$, $z \in \bd$. Clearly $f_\eta(z)= \sum\limits_{k=0}^{\infty} \langle a_k,\eta \rangle_\cle z^k$. Hence $f_\eta \in H^2 (\bd)$, as
\[
\sum |\langle a_k, \eta \rangle |^2 \leq \| \eta\|^2 \sum \|a_k\|^2 =  \|\eta\|^2 \|f\|^2 < \infty,
\]
by the Cauchy-Schwarz's inequality. Moreover, if $z \in \bt$, then
\[
f_\eta(z) = \lim_{r\rightarrow 1^-} f_\eta(r z) = \lim_{r\rightarrow 1^-} \langle f(rz), \eta \rangle = \langle f(z), \eta \rangle,
\]
which implies $f_\eta(z) = 0$ on $E$ for all $\eta \in \cle$. If $m(E) > 0$, then by the classical F. and M. Riesz theorem, $f_\eta = 0$ for each $\eta \in \cle$, and hence $f =0$.
\end{proof}

We are now ready to prove the following theorem:

\begin{theorem}\label{mult op Vector valued}
Let $\Phi \in H^\infty_{\mathscr{B}(\cle)}(\bd)$.

(i) If $M_{\Phi} \in \NA$, then $\|{\Phi}(z)\| =\|\Phi\|_{\infty}$, $z \in \mathbb{T}$ a.e.

(ii) If $\cle = \mathbb{C}$, then $M_{\Phi} \in \NA$ if and only if $\frac{1}{\|\Phi\|_{\infty}} \Phi$ is inner.
\end{theorem}
\begin{proof}
(i) Suppose $M_{\Phi}$ attains its norm at $f\in H^2_{\cle}(\bd)$. Then $\|M_{\Phi}f\|= \|M_\Phi\| \|f\| = \|\Phi\|_\infty \|f\|$ implies that
\[
\int_{\bt} \|\Phi(z) f(z)\|^2 \, dm(z) = \int_{\bt} \|\Phi\|^2_\infty \|f(z)\|^2 \, dm(z),
\]
and hence
\[
\int_{\bt} \|\Phi\|^2_\infty \|f(z)\|^2 \, dm(z) \geq \int_{\bt} \|\Phi(z)\|^2 \|f(z)\|^2 \, dm(z) \geq \int_{\bt} \|\Phi\|^2_\infty \|f(z)\|^2 \, dm(z),
\]
whence $\int_{\bt} \|\Phi\|^2_\infty \|f(z)\|^2 \, dm(z) = \int_{\bt} \|\Phi(z)\|^2 \|f(z)\|^2 \, dm(z)$. Lemma \ref{Risez thm vector valued} then implies that $\|{\Phi}(z)\| =\|\Phi\|$, $z \in \mathbb{T}$ a.e. as desired.

\noindent (ii) In view of part (i), it is enough to observe that $M_{\frac{1}{\|\Phi\|_{\infty}} \Phi}$ is an isometry whenever  $\frac{1}{\|\Phi\|_{\infty}} \Phi$ is inner.
\end{proof}

The converse of Theorem \ref{mult op Vector valued}(i) does not hold:

\begin{example}\label{not NA eg}
In the setting of Proposition \ref{mult op Vector valued}, consider $\cle =\ell^2(\mathbb{N})$ and the compact operator $K:\ell^2(\mathbb{N})\rightarrow\ell^2(\mathbb{N})$ defined by
\[
K\Big(\{\alpha_n\}_{n=1}^\infty\Big) = \Big\{\alpha_1, \frac{\alpha_2}{2}, \frac{\alpha_3}{3}, \ldots \Big\}.
\]
Note that $I - K \notin \NA$. Indeed, for any non-zero sequence $\{\alpha_n\}_{n=1}^\infty \in \ell^2(\mathbb{N})$, we have
\[ \|(I-K)\{\alpha_1, \alpha_2, \alpha_3,\ldots\}\|^2=\underset{n=1}{\overset{\infty}{\sum}}\left(1-\frac{1}{n}\right)^2 |\alpha_n|^2 < \underset{n=1}{\overset{\infty}{\sum}}|\alpha_n|^2.
\]
Define the constant function  $\Phi: \mathbb{D} \rightarrow \mathscr{B}(\cle)$ by $\Phi(z)=I-K$, $z \in \bd$. Clearly $\|M_\Phi\| = \|\Phi\| = \|I-K\| = 1$. We shall show that $M_{\Phi} \notin \NA$. Suppose towards a contradiction that $M_{\Phi} \in \NA$. Then there exists a non-zero $f$ in $H^2_{\cle}(\mathbb{D})$ such that $\|M_{\Phi}f\|=\|f\|$, and so
\[
\int_{\bt} \Big(\|f(z)\|^2 - \|(I-K)f(z)\|^2 \Big) dm(z) = 0.
\]
Since $\|I-K\|=1$, we have
\[
\|(I-K)f(z)\| = \|f(z)\| \qquad (z \in \bt \, a.e.).
\]
However, $I - K \notin \NA$, which is a contradiction.
\end{example}

\begin{example}\label{eg: back shift}
Now we comment on the inner function property of $\Phi$ in the statement of Theorem \ref{mult op Vector valued}. There, unlike the scalar case, if $\Phi \in H^\infty_{\mathscr{B}(\cle)}(\bd)$ with $\|\Phi\|=1$ and $M_\Phi \in \NA$, then $\Phi$ need not be inner. For instance, consider the backward shift $S^*$ on $\cle=\ell^2(\mathbb{N})$, that is
\[
S^* e_n =
\begin{cases} e_{n-1} & \mbox{if}~ n \geq 2
\\
0 & \mbox{if}~ n=1, \end{cases}
\]
 and define the constant function $\Phi(z) = S^*$, $z \in \bd$.
Of course, $\Phi(z) = S^*$ for all $z \in \bt$, and hence, it follows that $\Phi$ is not inner. On the other hand, define $f \in H^2_{\cle}(\bd)$ by $f(z) = e_2$, $z \in \bd$, where $e_2(i) = \delta_{2,i}$. Then $\|f\|=1$ and
\[
\|M_\Phi f\|^2 =  \int_{\bt} \|\Phi(z) f(z)\|^2 \, dm(z) = \int_{\bt} \|S^*e_2\|^2 \, dm(z) = 1 = \|\Phi\|^2,
\]
and hence $M_\Phi \in \NA$.
\end{example}

We conclude the paper with norm attaining Laurent operators. In this setting, the results and the ideas are similar to that of Theorem \ref{mult op Vector valued}. We first illustrate the scalar case: Let $\vp \in L^\infty(\bt)$. Suppose $f \neq 0$ in $L^2(\bt)$ satisfy $\|L_\vp f\| = \|\vp\|_\infty \|f\|_{L^2(\bt)}$. Then
\[
\left( \|\vp\|^2_\infty - |\vp (z)|^2 \right) |f(z)|^2=0 \qquad (z \in \bt \; a.e.).
\]
If $E = \{z \in \bt: f(z)= 0\}$, then $m(E) = 0$, and hence $|\vp (z)| = \|\vp\|_\infty$ for all $z \in E$, where $E \subseteq \bt$ and $m(E) > 0$. Conversely, if $m(E) > 0$ and $|\vp (z)| = \|\vp\|_\infty$ for all $z \in E$, then $M_\vp$ attain its norm at $\chi_E$. This proves the following:

\begin{proposition}\label{NA scalar Laurent }
Let $\vp \in L^\infty(\bt)$. Then $L_\vp \in \NA$ if and only if there exists a measurable set $A\subseteq \bt$ such that $m(A) > 0$ and $|\vp (z)| = \|\vp\|_\infty$ for all $z \in A$.
\end{proposition}

A similar (but one directional as in Theorem \ref{mult op Vector valued}) statement is valid for operator-valued Laruant operators: Let $\Phi\in L^\infty_{\mathscr{B}(\cle)}(\bt)$. If $L_\Phi \in \NA$, then there exists a measurable set $A \subseteq \bt$ such that $m(A) > 0$ and $\|\Phi(z)\| = \|\Phi\|$, $z \in A$. Again, the converse fails to hold: Example \ref{not NA eg} serves the purpose.

\vspace{0.3in}

\noindent\textit{Acknowledgement:} We are grateful to the referee for a careful reading of the paper and useful suggestions. We are especially thankful to the referee for suggesting the elegant proof of Theorem \ref{prop: NA idempotent}. The research of first and second author is supported by the Theoretical Statistics and Mathematics Unit, Indian Statistical Institute, Bangalore Centre, India. The third author is supported in part by the Mathematical Research Impact Centric Support, MATRICS (MTR/2017/000522), and Core Research Grant (CRG/2019/000908), by SERB, Department of Science \& Technology (DST), and NBHM (NBHM/R.P.64/2014), Government of India. The research of the fourth author is supported by the NBHM postdoctoral fellowship, Department of Atomic Energy (DAE), Government of India (File No: 0204/3/2020/R$\&$D-II/2445).


\begin{thebibliography}{99}
\bibitem{Acosta}
M. D. Acosta, R. M. Aron, D. Garc\'{\i}a, M. Maestre,  \emph{The {B}ishop-{P}helps-{B}ollob\'{a}s theorem for operators}, J. Funct. Anal. 254 (2008), no. 11, 2780--2799.
	
\bibitem{Ando}
T. Ando, \emph{Unbounded or bounded idempotent operators in Hilbert space}, Linear Algebra Appl. 438 (2013), 3769--3775.

\bibitem{Hari}
H. Bercovici,  {\em Operator theory and arithmetic in $H^\infty$}, Mathematical Surveys and Monographs, 26. American Mathematical Society, Providence, RI, 1988.

\bibitem{Bottcher}
A. Böttcher and I. M. Spitkovsky, \emph{A gentle guide to the basics of two projections theory}, Linear Algebra Appl. 432 (2010), no. 6, 1412-1459.

\bibitem{BD} A. Brown and R. Douglas, \emph{Partially isometric Toeplitz operators}, Proc. Amer. Math. Soc. 16 (1965), 681--682.


\bibitem{Buc1}
D. Buckholtz, {\em Hilbert space idempotents and involutions}, Proc. Amer. Math. Soc. 128 (2000), 1415–1418.

\bibitem{Buc}
D. Buckholtz, {\em Inverting the difference of Hilbert space projections}, Amer. Math. Monthly, 104 (1997), 60–61.



\bibitem{Neves} X. Carvajal and W. Neves, \emph{Operators that achieve the norm}, Integral Equations Operator Theory, 72 (2012), 179--195.

\bibitem{FELDMAN} I. Feldman, N. Krupnik\ and\ A. Markus, On the norm of polynomials of two adjoint projections, Integral Equations Operator Theory {\bf 14} (1991), no.~1, 69--91.

\bibitem{GR}
S. Garcia and W. Ross, {\em A non-linear extremal problem on the Hardy space}, Comput. Methods Funct. Theory, 9 (2009), 485–524.


\bibitem{HH}
H. Helson, {\em Lectures on invariant subspaces}, Academic Press, New York-London 1964.

\bibitem{L}
J. Lindenstrauss, \emph{On operators which attain their norm},
Israel J. Math. 1 (1963), 139--148.

\bibitem{M}
M. Mart\'{\i}n, \emph{Norm-attaining compact operators},  J. Funct. Anal. 267 (2014), no. 5, 1585--1592.

\bibitem{SKP}
S. Pandey, {\em A spectral characterization of absolutely norming operators on s.n. ideals}, Oper. Matrices, 11 (2017), 845–873.


\bibitem{PAL} S. Pandey and V. Paulsen,  \emph{A spectral characterization of $\mathcal{AN}$ operators}, J. Aust. Math. Soc. 102 (2017), 369--391.

\bibitem{Ramesh}
G. Ramesh, {\em Structure theorem for $\mathcal{A} \mathcal{N}$-operators}, J. Aust. Math. Soc. 96 (2014), 386–395.

\bibitem{Sarason}
D. Sarason, {\em Generalized interpolation in $H^\infty$}, Trans. Amer. Math. Soc. 127 (1967), 179–203.





\bibitem{NF} B. Sz.-Nagy and C. Foia\c{s}, Harmonic analysis of operators on Hilbert space. Translated from the French and revised North-Holland Publishing Co., Amsterdam-London; American Elsevier Publishing Co., Inc., New York; Akad\'{e}miai Kiad\'{o}, Budapest 1970.

\bibitem{Szyld}
D. B. Szyld, {\em The many proofs of an identity on the norm of oblique projections},
Numer. Algorithms 42 (2006), no.3-4, 309-323.



\end{thebibliography}
\end{document}